\documentclass[preprint,12pt]{elsarticle}

\usepackage{amssymb}
\usepackage{hyperref}

\newtheorem{theorem}{Theorem}
\newtheorem{lemma}{Lemma}[section]
\newtheorem{proposition}[theorem]{Proposition}
\newproof{proof}{Proof}

\journal{Annals of Pure and Applied Logic}
\usepackage{proof,rotating}

\newcommand{\LL}{\mathbf{L}}
\newcommand{\LLS}{\mathbf{L}\!^*}
\newcommand{\EL}{\mathbf{EL}}
\newcommand{\ELS}{\mathbf{EL}\!^*}
\newcommand{\ELw}{\EL^{\mathrm{wk}}}

\newcommand{\ELs}{\EL\!^{-}}
\newcommand{\ELsu}{\ELs\!\!_{{}\SL{}}}

\newcommand{\ELm}{\EL^{\mathrm{mk}}}
\newcommand{\ELd}{\EL^{\!\dagger}}
\newcommand{\ELdd}{\EL^{\!\ddagger}}

\newcommand{\LLu}{\LL_{{}\SL{}}}

\newcommand{\WEAK}{\mathrm{weak}}
\newcommand{\CONTR}{\mathrm{contr}}
\newcommand{\EXC}{\mathrm{perm}}

\newcommand{\CUT}{\mathrm{cut}}
\newcommand{\RED}{\mathrm{red}}
\newcommand{\SUBST}{\mathrm{subst}}
\newcommand{\MON}{\mathrm{mon}}

\newcommand{\Ac}{\mathcal{A}}
\newcommand{\Gc}{\mathcal{G}}

\newcommand{\GA}{\Gamma\!_{\Ac}}

\newcommand{\Tp}{\mathrm{Tp}}

\newcommand{\BS}{\mathop{\backslash}}
\newcommand{\SL}{\mathop{/}}

\newcommand{\mk}[2]{{#1}_{(#2)}}

\begin{document}

\begin{frontmatter}



\title{Reconciling Lambek's Restriction, Cut-Elimination, and Substitution in the Presence of Exponential Modalities}

\author{Max Kanovich}
\address{University College London and\\
National Research University Higher School of Economics (Moscow)}
\author{Stepan Kuznetsov}
\address{Steklov Mathematical Institute (Moscow) and\\
National Research University Higher School of Economics (Moscow)}
\author{Andre Scedrov}
\address{University of Pennsylvania and\\
National Research University Higher School of Economics (Moscow)}


%\address{}

\begin{abstract}
The Lambek calculus can be considered as a version of non-com\-mu\-ta\-tive
intuitionistic
 linear logic. One of the interesting features of
the Lambek calculus is the so-called ``Lambek's restriction,'' that is, the
 antecedent of any provable sequent should be non-empty.
In this paper we discuss ways of extending the Lambek calculus with
the linear
logic exponential modality while keeping Lambek's restriction.
Interestingly enough, we show that for any system equipped with a
reasonable exponential modality
the following holds: if the system enjoys cut elimination and substitution
to the full extent, then the system necessarily violates Lambek's restriction.
Nevertheless, we show that two of the three conditions can be implemented. Namely, we
design a system with Lambek's restriction and cut elimination
and another system with Lambek's restriction and substitution. 
For both calculi
we prove that they are undecidable, even if we take only one
of the two divisions provided by the Lambek calculus.
The system with cut elimination and substitution 
and without Lambek's restriction 
is folklore and known to be undecidable.
\end{abstract}


\end{frontmatter}


\section{Introduction}

\subsection{The Lambek Calculus}

The Lambek calculus was introduced by J. Lambek in~\cite{Lambek58} for
mathematical description of natural language syntax by means of so-called
{\em Lambek categorial (type-logical) grammars} (see, 
for example,~\cite{Carpenter98,Morrill11,MootRetore12}). In Lambek
grammars, syntactic categories are represented by logical formulae involving
three connectives: the {\em product} (corresponds to concatenation of words)
and two {\em divisions} (left and right), and syntactic correctness of
natural language expressions corresponds to derivability in  the Lambek calculus.

For simplicity, in this paper we discuss only the product-free fragment of
the Lambek calculus. First we consider not the Lambek calculus $\LL$~\cite{Lambek58}, but its variant $\LLS$~\cite{Lambek61}. The difference between $\LL$ and $\LLS$ is explained in
the end of this introductory section (see ``Lambek's Restriction'').

$\LLS$ is a substructural logic, and here we formulate it as a Gentzen-style sequent
calculus. Formulae of $\LLS$ are called {\em types} and are built from {\em variables,} or
{\em primitive types} ($p$, $q$, $r$, $p_1$, $p_2$, \dots) using two binary connectives:
$\BS$ {\em (left division)} and $\SL$ {\em (right division)}.
Types are denoted by capital Latin letters; finite
(possibly empty) linearly ordered sequences of types by capital Greek ones.
$\Lambda$ stands for the empty sequence. The Lambek calculus derives
objects called {\em sequents} of the form $\Pi \to A$, where
the {\em antecedent} $\Pi$ is a linearly ordered sequence of types and
the {\em succedent} $A$ is a type.
%called
%the {\em antecedent} (left-hand side), and $A$ is called the {\em 
%succedent} (right-hand side) of the sequent.

The axioms of $\LLS$ are all sequents $A \to A$, where
$A$ is a type, and the rules of inference are as follows:

\begin{center}
\begin{tabular}{c@{\qquad}c}
$\infer[(\to\BS)]{\Pi \to A \BS B}{A, \Pi \to B}$ &
$\infer[(\BS\to)]{\Delta_1, \Pi, A \BS B, \Delta_2 \to C}
{\Pi \to A & \Delta_1, B, \Delta_2 \to C}$\\[6pt]
$\infer[(\to\SL)]{\Pi \to B \SL A}{\Pi, A \to B}$ &
$\infer[(\SL\to)]{\Delta_1, B \SL A, \Pi, \Delta_2 \to C}
{\Pi \to A & \Delta_1, B, \Delta_2 \to C}$
\end{tabular}
\end{center}

For $\LLS$ and other calculi introduced later in this paper,
we do not include cut as an official rule of the system.
However, the cut rule of the following non-commutative form
$$
\infer[(\CUT)]{\Delta_1, \Pi, \Delta_2 \to B}
{\Pi \to A & \Delta_1, A, \Delta_2 \to B}
$$
is admissible in $\LLS$~\cite{Lambek61}.

By $\LLS_{{}\SL{}}$ (resp., $\LLS_{{}\BS{}}$) we denote the fragment of $\LLS$ with
only the right (resp., left) division connective. Due to the subformula property,
these fragments are obtained from the full calculus simply by restricting
the set of rules.

\subsection{The Exponential Modality}

We see that $\LLS$ lacks structural rules (except for the
implicit rule of associativity).

$\LLS$ can be conservatively embedded~\cite{Abrusci90,Yetter90}
 into a non-commutative, intuitionistic or cyclic,
 variant of Girard's~\cite{Girard87} linear logic.
In the spirit of linear logic connectives, the Lambek
calculus can be extended with
the {\em exponential} unary connective that enables
structural rules (weakening, contraction, and commutativity) in a controlled way.

We'll denote this extended calculus by $\ELS$. Types of $\ELS$ are
built from variables using two binary connectives ($\BS$ and $\SL$) and a
unary one, ${!}$, called the {\em exponential}, or, colloqually, {\em ``bang.''}
If $\Gamma = A_1, \dots, A_k$, then by ${!}\Gamma$ we denote
the sequence ${!}A_1, \dots, {!}A_k$.
$\ELS$ is obtained from $\LLS$ by adding the following rules:

\begin{center}
\begin{tabular}{c@{\qquad}c}
$\infer[(!\to)]{\Delta_1,!A,\Delta_2 \to B}{\Delta_1,A,\Delta_2 \to B}$ &
$\infer[(\to {!})]{!\Gamma \to {!}A}{!\Gamma \to A}$\\[3pt]
$\infer[(\WEAK)]{!A,\Delta \to B}{\Delta \to B}$ &
$\infer[(\CONTR)]{!A,\Delta \to B}{!A, !A, \Delta\to B}$\\[3pt]
$\infer[(\EXC_1)]{\Delta_1,!A,B,\Delta_2 \to C}
{\Delta_1,B,!A,\Delta_2 \to C}$ &
$\infer[(\EXC_2)]{\Delta_1,A,!B,\Delta_2 \to C}
{\Delta_1,!B,A, \Delta_2 \to C}$
\end{tabular}
\end{center}

The following theorem is proved in~\cite{Kanovich93} and~\cite{deGroote05} and
summarized in~\cite{Kanazawa99}.
A weaker result that $\ELS$ with the product and two
divisions is undecidable follows
from~\cite{Lincoln92,Kanovich16}.

\begin{theorem}\label{Th:undecELS}
The derivability problem for $\ELS$ is undecidable.
\end{theorem}

\subsection{Lambek's Restriction}

The original Lambek calculus $\LL$~\cite{Lambek58} differs
from the presented above in one detail: in $\LL$, sequents with
empty antecedents are not permitted. This restriction applies not
only to the final sequent, but to all ones in the derivation. Thus,
for example, the sequent $(q \BS q) \BS p \to p$ is derivable in
$\LLS$, but not in $\LL$, though its antecedent is not empty
(but the $\LLS$-derivation involves the sequent $\to q \BS q$ with
an empty antecedent). 
Further we shall use the term
{\em Lambek's restriction} for this special constraint.
Actually, Lambek's restriction in $\LLS$ could potentially be violated
only by application of the $(\to\BS)$ and $(\to\SL)$ rules, therefore
$\LL$ can be obtained from $\LLS$ by adding the constraint 
``$\Pi$ is non-empty'' to these two rules.

At first glance, Lambek's restriction looks strange and formal,
but it is highly motivated by linguistic applications.

{\it Example 1.}~\cite[2.5]{MootRetore12}
In syntactic formalisms based on the Lambek calculus,
Lambek types denote syntactic categories. 
Let 
$n$ stand for ``noun phrase,'' then $n \SL n$ 
is going to be a ``noun modifier'' (it can be combined
with a noun phrase on the right producing a new, more complex
noun phrase: $\LL \vdash n \SL n, n \to n$), i.e.\ an adjective.
Adverbs, as adjective modifiers, receive the type $(n \SL n) \SL
(n \SL n)$. Now one can derive the sequent
\mbox{$(n \SL n) \SL (n \SL n), n \SL n, n \to n$} and therefore establish
that, say, ``very interesting book'' is a valid noun phrase (belongs to
syntactic category $n$). However, in $\LLS$ 
one can also derive
\mbox{$(n \SL n) \SL (n \SL n), n \to n$}, where the antecedent describes
syntactic constructions like ``very book,'' that in fact are not
correct noun phrases. 

This example shows that, for linguistic
purposes, $\LL$ is more appropriate than~$\LLS$.

Suprisingly, however, it is not so straightforward to add
 the exponential to $\LL$ or
to impose Lambek's restriction on $\ELS$.
In Sections~2--7 we discuss several ways how to do this,
define a number of the corresponding calculi, prove their properties,
and discuss some issues connected with these calculi.

In Section~\ref{S:undec} we state and prove undecidability results
for calculi defined earlier; in Section~\ref{S:Buszko} we prepare
the techniques then used in Section~\ref{S:undec}.
Finally, Section~\ref{S:conclusion}
contains general discussion of the results and possible directions
of future work.

\subsection{Is it Possible to Maintain Three Properties Together:
Lambek's Restriction, Cut Elimination, and Substitution?}
 
No. We show (Theorems~\ref{Th:ELd} and~\ref{Th:ELdd}) that for any
system equipped a reasonable ${!}$ the following holds:
if the system enjoys cut elimination and substitution in full
extent, then this system necessarily violates Lambek's restriction.
(More precisely, adding one formula starting with ${!}$ to the antecedent
allows $\LLS$ derivations inside such a system.)

Nevertheless, any two of these three properties are realisable in the
calculi defined below, namely:
\begin{itemize}
\item $\ELs$ (Section~\ref{S:ELs}) has Lambek's restriction and cut elimination, but
substitution only for formulae without ${!}$;
\item $\ELm$ (Section~\ref{S:ELm}) has Lambek's restriction and substitution in the full form,
but the cut rule is admissible only for formulae without ${!}$;
\item finally, $\ELS$ enjoys both substitution and cut, but without Lambek's
restriction.
\end{itemize}

\section{Imposing Lambek's Restriction on $\ELS$: the 1st Approach, $\ELw$}

The first, na{\"{\i}}ve way of imposing Lambek's restriction on
$\ELS$ is to restrict only rules $(\to\BS)$ and $(\to\SL)$ in the
same way as it is done in $\LL$. Notice that
 all other rules, including rules for the exponential,
preserve the non-emptiness of the antecedent. %, and to
Denote the calculus
by $\ELw$.

However, such a restriction does not change things significantly,
since the following lemma provides the non-emptiness of the antecedent
for free:
\begin{lemma}\label{Lm:vanish}
Let $p$ be a variable not occurring in a sequent $\Gamma \to A$. Then
$$\ELS \vdash \Gamma \to A \iff
\ELw \vdash {!}p, \Gamma \to A.$$
\end{lemma}

This lemma shows that $\ELS$-derivations can be enabled
in $\ELw$ by an easy technical trick. Therefore, Theorem~\ref{Th:undecELS}
implies immediately that $\ELw$ is undecidable.

\begin{lemma}
$\ELw \vdash {!}B, \Gamma \to A \iff
\ELS \vdash {!}B, \Gamma \to A$.
\end{lemma}

These two lemmas are proved by induction on the derivations
(recall that $(\CUT)$ is not included in the calculi).

Thus, Lambek's restriction in $\ELw$ vanishes as soon as
the antecedent contains a formula with ${!}$ as the main connective.
And, unfortunately, this acts non-locally: once ${!}A$ appears somewhere
in the antecedent, one can freely derive unwanted things like ``very book''
(see Example~1 above).

\section{Imposing Lambek's Restriction on $\ELS$: the 2nd Approach, $\ELs$}\label{S:ELs}
To overcome the ability of ${!}B$ to mimic the empty antecedent,
we impose more radical restrictions by constructing the following
calculus $\ELs$.

Any formula not of the form ${!}B$ is called a {\em non-bang-formula.}
(A non-bang-formula is allowed to have proper subformulae with ${!}$.)
Now $\ELs$ is defined by the following axioms and
rules:

$$
A \to A
$$

$$
\infer[(\to\BS)\mbox{, where $\Pi$ contains a non-bang-formula}]
{\Pi \to A \BS B}{A, \Pi \to B}
$$

$$
\infer[(\to\SL)\mbox{, where $\Pi$ contains a non-bang-formula}]
{\Pi \to B \SL A}{\Pi, A \to B}
$$

$$
\infer[(\BS\to)]{\Delta_1, \Pi, A \BS B, \Delta_2 \to C}
{\Pi \to A & \Delta_1, B, \Delta_2 \to C}
\qquad
\infer[(\SL\to)]{\Delta_1, B \SL A, \Pi, \Delta_2 \to C}
{\Pi \to A & \Delta_1, B, \Delta_2 \to C}
$$

$$
\infer[({!}\to)\mbox{, where $\Delta_1, \Delta_2$ contains a
non-bang-formula}]
{\Delta_1, {!}A, \Delta_2 \to B}{\Delta_1, A, \Delta_2 \to B}
$$

$$
\infer[(\WEAK)]{{!}A, \Delta \to B}{\Delta \to B}
\qquad
\infer[(\CONTR)]{{!}A, \Delta \to B}{{!}A, {!}A, \Delta \to B}
$$

$$\infer[(\EXC_1)]{\Delta_1,!A,B,\Delta_2 \to C}
{\Delta_1,B,!A,\Delta_2 \to C} \qquad
\infer[(\EXC_2)]{\Delta_1,A,!B,\Delta_2 \to C}
{\Delta_1,!B,A, \Delta_2 \to C}
$$

Note that in the $(\to{!})$ rule of $\ELS$ all the formulae in the
antecedent are of the form ${!}B$. Therefore there is no
$(\to{!})$ rule in $\ELs$. Also note that the cut rule is not officially included in $\ELs$;
in the next section we prove that it is admissible.

\begin{lemma}\label{Lm:ELsRestr}
If $\Pi \to A$ is derivable in $\ELs$ and $\Pi$ is of the form ${!}\Gamma$, then
$A$ is of the form ${!}B$, such that ${!}B$ appears in ${!}\Gamma$.
\end{lemma}

\begin{proof}
A sequent with no non-bang-formula in the antecedent
can be obtained only by applying $(\WEAK)$, $(\EXC)$,
$(\CONTR)$ (in any combination) to an axiom of the form
${!}B \to {!}B$. All other rules either introduce non-bang-formulae to the left,
or explicitly require their existence.
\qed
\end{proof}

Now Lambek's restriction in $\ELs$ is stated in the following
way: {\em in a non-trivial derivable sequent $\Pi \to A$ the antecedent $\Pi$ 
should contain at least one non-bang-formula.}

\section{The Cut Rule in $\ELs$}

\begin{lemma}\label{Lm:nonbang_inherit}
If $\Pi \to A$ is derivable in $\ELs$ and $A$ is a non-bang-formula,
then $\Pi$ necessarily contains a non-bang-formula.
\end{lemma}

\begin{proof}
Immediately from Lemma~\ref{Lm:ELsRestr}.
\qed
\end{proof}

\begin{theorem} The cut rule
$$
\infer[(\CUT)]
{\Delta_1, \Pi, \Delta_2 \to C}
{\Pi \to A & \Delta_1, A, \Delta_2 \to C}
$$
is admissible in $\ELs$.
\end{theorem}

\begin{proof}
We proceed by double induction. We consider a number of cases, and in
each of them the cut either disappears, or is replaced by cuts with simpler
cut formulae ($A$), or is replaced by a cut for which the depth of at least
one derivation tree of a premise ($\Pi \to A$ or $\Delta_1, B, \Delta_2 \to C$) is less
than for the original cut, and the other premise derivation and the cut formula
remain the same. Thus by double induction (on the outer level---on the complexity
of $A$, on the inner level---on the sum of premise derivation tree depths) we
get rid of the cut.

{\bf Case~1:} One of the premises of the cut rule is the axiom (a sequent of
the form $A \to A$). Then the cut disappears, since its other premise concides
with the goal sequent.

{\bf Case~2:} $A$ is not the formula that is introduced by the lowermost rule in
the derivation of one of the premises of the cut. (The term {\em 
``formula introduced by a rule''} here means the following: rules $(\SL\to)$, $(\BS\to)$, \mbox{$(\to\SL)$,}
$(\to\BS)$, $({!}\to)$, and $(\to{!})$ introduce the formula that includes the new
connective; $(\WEAK)$ and $(\CONTR)$ introduce ${!}A$ involved in these rules;
$(\EXC_1)$ and $(\EXC_2)$ do not introduce anything.)

In this case $(\CUT)$ can be interchanged
with that lowermost rule. Many subcases arise here, depending on the particular form
of the rule interchanged with $(\CUT)$, but they are all handled similarly. Below we show
only the most interesting situations,
when we interchange $({!}\to)$ or $(\to\SL)$ ($(\to\BS)$ is symmetric) with $(\CUT)$.
In these transformations Lambek's restriction imposed on these rules could potentially get violated 
after the exchange with $(\CUT)$ (we show that this does not happen).

{\it Transformation~1.}
$$
\infer[(\CUT)]
{\Delta_1, \Pi', {!}D, \Pi'', \Delta_2 \to C}
{\infer[({!}\to)]{\Pi', {!}D, \Pi'' \to A}{\Pi', D, \Pi'' \to A} &
\Delta_1, A, \Delta_2 \to C}
$$
\centerline{\raisebox{10pt}{\turnbox{270}{$\leadsto$}}}
%\vskip 10pt
$$
\infer[({!}\to)]
{\Delta_1, \Pi', {!}D, \Pi'', \Delta_2 \to C}
{\infer[(\CUT)]{\Delta_1, \Pi', D, \Pi'', \Delta_2 \to C}
{\Pi', D, \Pi'' \to A &
\Delta_1, A, \Delta_2 \to C}}
$$

{\it Transformation~2.}
$$
\infer[(\CUT)]
{\Delta_1, \Pi, \Delta'_2, {!}D, \Delta''_2 \to C}
{\Pi \to A & 
\infer[({!}\to)]{\Delta_1, A, \Delta'_2, {!}D, \Delta''_2 \to C}
{\Delta_1, A, \Delta'_2, D, \Delta''_2 \to C}}
$$
\centerline{\raisebox{10pt}{\turnbox{270}{$\leadsto$}}}
$$
\infer[({!}\to)]
{\Delta_1, \Pi, \Delta'_2, {!}D, \Delta''_2 \to C}
{\infer[(\CUT)]{\Delta_1, \Pi, \Delta'_2, D, \Delta''_2 \to C}
{\Pi \to A & \Delta_1, A, \Delta'_2, D, \Delta''_2 \to C}}
$$

{\it Transformation~3.}
$$
\infer[(\CUT)]
{\Delta_1, \Pi, \Delta_2 \to C_2 \SL C_1}
{\Pi \to A &
\infer[(\to\SL)]
{\Delta_1, A, \Delta_2 \to C_2 \SL C_1}
{\Delta_1, A, \Delta_2, C_1 \to C_2}}
$$
\centerline{\raisebox{10pt}{\turnbox{270}{$\leadsto$}}}
$$
\infer[(\to\SL)]
{\Delta_1, \Pi, \Delta_2 \to C_2 \SL C_1}
{\infer[(\CUT)]{\Delta_1, \Pi, \Delta_2, C_1 \to C_2}
{\Pi \to A & \Delta_1, A, \Delta_2, C_1 \to C_2}}
$$

In Transformation~1, the existence of a non-bang-formula in $\Pi'$ or $\Pi''$ implies
its existence in the larger context $\Delta_1, \Pi', \Pi'', \Delta_2$, thus application
of $({!}\to)$ is legal. For Transformations~2 and~3, if the non-bang-formula guaranteed
by Lambek's restriction (which was indeed valid before the transformation) 
is $A$ itself, then a non-bang-formula
also appears in $\Pi$ by Lemma~\ref{Lm:nonbang_inherit}. Otherwise we can take the
same non-bang-formula as before the transformation.

{\bf Case 3:} $A$ is introduced by the lowermost rules both into
$\Pi \to A$ and into $\Delta_1, A, \Delta_2 \to C$.
Note that in $\ELs$ there is no rule that introduces a formula
of the form ${!}E$ to the succedent. Therefore
$A$ is either of the form $E \SL F$, or $F \BS E$. We consider only
the former, the latter is handled symmetrically. The derivation
is transformed in the following way:

{\it Transformation~4.}
$$
\infer[(\CUT)]
{\Delta_1, \Pi, \Gamma, \Delta_2 \to C}
{\infer[(\to\SL)]{\Pi \to E \SL F}{\Pi, F \to E}
 &
\infer[(\SL\to)]{\Delta_1, E \SL F, \Gamma, \Delta_2 \to C}
{\Gamma \to F & \Delta_1, E, \Delta_2 \to C}}
$$
\centerline{\raisebox{10pt}{\turnbox{270}{$\leadsto$}}}
$$
\infer[(\CUT)]
{\Delta_1, \Pi, \Gamma, \Delta_2 \to C}
{\infer[(\CUT)]{\Pi, \Gamma \to E}{\Gamma \to F & \Pi, F \to E}
& \Delta_1, E, \Delta_2 \to C}
$$

The new cut formulae, $E$ and $F$, are simpler than the original one, $E \SL F$.
\qed
\end{proof}

\section{Substitution Issues}

$\ELs$ inherits a bang-free substitution lemma:
derivability of a sequent is preserved if we replace
all occurrences of a variable $q$ with a formula $Q$ without ${!}$.

In the general case, however, 
$\ELs$ does not respect type substitution. For instance,
$p, {!}(p \BS q) \to q$ is derivable in $\ELs$, but
${!}r, {!}({!}r \BS q) \to q$ is not. Unfortunately,
this is not just a problem with this particular system,
but a general issue: as we show below, in the presence
of the exponential modality
any system enjoying admissibility of $(\CUT)$ and
general substitution lemma necessarily violates Lambek's restriction.

Let $\ELd$ be an arbitrary calculus, in the same language
as $\ELS$, satisfying the properties below.
(Note that these properties do not {\em define} the calculus in a unique way---we
rather talk about {\em a family} of possible `good' extensions of the Lambek
calculus.)

\begin{enumerate}
\item {\em Extension.} If a sequent is derivable in $\LL$, then it is derivable
in $\ELd$.
\item {\em Cut.} The cut rule of the form
$$
\infer[(\CUT)]
{\Delta_1, \Pi, \Delta_2 \to C}{\Pi \to A & \Delta_1, A, \Delta_2 \to C}
$$
is admissible in $\ELd$.
\item {\em Substitution.}
The following rule is admissible in $\ELd$:
$$
\infer[(\SUBST)]
{\Pi[q := Q] \to A[q := Q]}
{\Pi \to A}
$$
Here $q$ is a variable, $Q$ is a formula (possibly with ${!}$), and $[q := Q]$ denotes
 substitution of $Q$ for $q$.
\item {\em Monotonicity.}
The following rules are admissible in $\ELd$:
$$
\infer[(\MON_{\SL})]{B_1 \SL A_2 \to B_2 \SL A_1}{A_1 \to A_2 & B_1 \to B_2}
$$ $$
\infer[(\MON_{\BS})]{A_2 \BS B_1 \to A_1 \BS B_2}{A_1 \to A_2 & B_1 \to B_2}
$$

\item {\em Weakening, contraction, and permutation.}
The following rules are admissible in $\ELd$:
\begin{center}
\begin{tabular}{c@{\qquad}c}
$\infer[(\WEAK)]{!A,\Delta \to B}{\Delta \to B}$ &
$\infer[(\CONTR)]{!A,\Delta \to B}{!A, !A, \Delta\to B}$\\[3pt]
$\infer[(\EXC_1)]{\Delta_1,!A,B,\Delta_2 \to C}
{\Delta_1,B,!A,\Delta_2 \to C}$ &
$\infer[(\EXC_2)]{\Delta_1,A,!B,\Delta_2 \to C}
{\Delta_1,!B,A, \Delta_2 \to C}$
\end{tabular}
\end{center}

\item The rules
$$\infer[(\SL\to)]{\Delta_1, B \SL A, \Pi, \Delta_2 \to C}
{\Pi \to A & \Delta_1, B, \Delta_2 \to C}
\raisebox{7pt}{\mbox{\qquad and \qquad}}
\infer[(\BS\to)]{\Delta_1, \Pi, A \BS B, \Delta_2 \to C}
{\Pi \to A & \Delta_1, B, \Delta_2 \to C}
$$
are admissible in $\ELd$ without restrictions.

\item If $\Pi$ contains a formula without occurrences of ${!}$ (and therefore
is non-empty) and $B$ does not contain occurrences of ${!}$, then the rules
$$\infer[(\to\SL)]{\Pi \to B \SL A}{\Pi, A \to B}
\raisebox{7pt}{\mbox{\qquad and \qquad}}
\infer[(\to\BS)]{\Pi \to A \BS B}{A, \Pi \to B}
$$
are admissible in $\ELd$.
\end{enumerate}

Note that $(\CUT)$, $(\SUBST)$, and $(\MON)$ are admissible in $\LL$, therefore we want
them to keep valid in the extension. Weakening, contraction, and permutation are basic
rules for the exponential. Finally, the last two properties ensure that the version
of Lambek's restriction used in $\ELd$ does not forbid Lambek derivations in the presence
of the exponential modality.

Also note that by substitution we get the axiom $A \to A$ for arbitrary $A$, possibly
with occurrences of ${!}$.

Unfortunately, any calculus $\ELd$ with these 7 properties necessarily
violates Lambek's restriction:

\begin{lemma}\label{Lm:ELd}
If $\ELd$ satisfies properties 1--7, $A$ and $B$ do not contain \ ${!}$, and
$\ELd \vdash {!}q, A \to B$, then $\ELd \vdash {!}q \to A \BS B$ and 
\mbox{$\ELd \vdash
{!}q \to B \SL A$}.
\end{lemma}

\begin{proof}
$$
\infer[(\CUT)]{{!}q \to A \BS B}
{
\infer[(\CONTR)]{{!}q \to (A \SL {!}q) \BS B}
{\infer[(\SUBST)]{{!}q, {!}q \to (A \SL {!}q) \BS B}
{\infer[(\to\BS)]{q, {!}q \to (A \SL q) \BS B}
{\infer[(\SL\to)]{A \SL q, q, {!}q \to B}
{q \to q & \infer[(\EXC)]{A, {!}q \to B}{{!}q, A \to B}}}}}
& 
\infer[(\MON)]{(A \SL {!}q) \BS B \to A \BS B}
{\infer[(\to\SL)]{A \to A \SL {!}q}{
\infer[(\WEAK)]{A, {!}q \to A}{A \to A}} & B \to B}}
$$
The $\SL$ case is symmetric.
\qed
\end{proof}

\begin{theorem}\label{Th:ELd}
If $\Pi \to B$ is derivable in $\LLS$, and $\ELd$ satisfies properties 1--7,
then ${!}q, \Pi \to B$ is derivable in $\ELd$.
\end{theorem}

\begin{proof}
Induction on derivation length. For the axiom case we use the weakening rule
to add ${!}q$. Applications of Lambek rules are translated straightforwardly;
the only non-trivial case is $(\to\SL)$ and $(\to\BS)$ with an empty $\Pi$,
where we use Lemma~\ref{Lm:ELd}.
\qed
\end{proof}

One could think that this effect is due to the weakening rule
(this rule allows forcing the antecedent
to be non-empty). However, in the fragment with only one variable
a result like Theorem~\ref{Th:ELd} can be achieved without weakening.
Note that, in the view of~\cite{KanovichOneVar}, \cite{Kanovich95}, \cite[Chapter~3]{Hendriks},
\cite{Metayer}, and~\cite{KuznIGPL},
the one-variable fragment of the Lambek calculus is as powerful as
the full calculus with a countable set of variables.

\begin{theorem}\label{Th:ELdd}
Let $\ELdd$ satisfy properties 1--4, 6, and 7. Let also contraction and permutation (but
not weakening) rules
be admissible in $\ELdd$, and, in addition, let $\ELdd$ include the $({!}\to)$ rule of the form
$$
\infer[({!}\to),]
{\Delta_1, {!}A, \Delta_2 \to C}
{\Delta_1, A, \Delta_2 \to C}
$$
if $\Delta_1$ or $\Delta_2$ contains a formula without occurrences of \ ${!}$ (this is the
strongest version of Lambek's restriction that could be imposed on this rule).
In this case, if $\Pi$ and $B$ contain only one variable $p$ and do not contain occurrences of \ ${!}$,
and $\Pi \to B$ is derivable in $\LLS$, then $\Pi, {!}(p \BS p) \to B$ is derivable in $\ELdd$.
\end{theorem}

\begin{proof}
First we state an easy technical lemma:
\begin{lemma}\label{Lm:ELdd}
If $A$ contains only one variable $p$ and no occurrences of \ ${!}$, then\linebreak $\ELdd \vdash A, {!}(p \BS p) \to A$.
\end{lemma}

\begin{proof}
Induction on the complexity of $A$. If $A = p$, then $p, {!}(p \BS p) \to p$ is derived as
follows:
$$
\infer[({!}\to)]
{p, {!}(p \BS p) \to p}
{\infer[(\BS\to)]{p, p \BS p \to p}
{p \to p & p \to p}}
$$

For $A = A_2 \SL A_1$ and $A = A_1 \BS A_2$ we use the following derivations:
$$
\infer[(\to\SL)]
{A_2 \SL A_1, {!}(p \BS p) \to A_2 \SL A_1}
{\infer[(\EXC)]{A_2 \SL A_1, {!}(p \BS p), A_1 \to A_2}
{\infer[(\SL\to)]{A_2 \SL A_1, A_1, {!}(p \BS p) \to A_2}
{A_1 \to A_1 & A_2, {!}(p \BS p) \to A_2}}}
\qquad
\infer[(\to\BS)]
{A_1 \BS A_2, {!}(p \BS p) \to A_1 \BS A_2}
{\infer[(\BS\to)]{A_1, A_1 \BS A_2, {!}(p \BS p) \to A_2}
{A_1 \to A_1 & A_2, {!}(p \BS p) \to A_2}}
$$

$A_1 \to A_1$ is an axiom of $\LL$ (and, by property~1,
is derivable in $\ELdd$). $A_2, {!}(p \BS p) \to A_2$ is derivable
by induction hypothesis.
\qed
\end{proof}

Then we proceed by induction on derivation. The axiom $A \to A$ becomes a
derivable sequent $A, {!}(p \BS p) \to A$ (Lemma~\ref{Lm:ELdd}). Now the only non-trivial
case is to simulate $(\to\SL)$ and $(\to\BS)$ with an empty $\Pi$:

$$\small
\infer[(\CUT)]{{!}(p \BS p) \to A \BS B}
{
\infer[(\CONTR)]
{{!}(p \BS p) \to (A \SL {!}(p \BS p)) \BS B}
{
\infer[(\SUBST)]{{!}(p \BS p), {!}(p \BS p) \to (A \SL {!}(p \BS p)) \BS B}
{
\infer[(\to\BS)]{p, {!}(p \BS p) \to (A \SL p)\BS B}
{
\infer[(\SL\to)]{A \SL p, p, {!}(p \BS p) \to B}
{p \to p & A, {!}(p \BS p) \to B}
}
}
}
&
\infer[(\MON)]{(A \SL {!}(p \BS p)) \BS B \to A \BS B}
{\infer[(\to\SL)]{A \to A \SL {!}(p \BS p)}{\infer{A, {!}(p \BS p) \to A}{\mbox{by Lemma~\ref{Lm:ELdd}}} }& B \to B}
}
$$

The $(\to\SL)$ case is handled symmetrically.
\qed
\end{proof}

Theorems~\ref{Th:ELd} and~\ref{Th:ELdd} show that there is no way
to add the exponential modality to the Lambek calculus preserving
Lambek's restriction,
admissibility of $(\CUT)$, and the substitution property at the same
time.

In the next section we describe another yet extension of $\LL$ with
the exponential modality. This extension features a version Lambek's restriction,
admits substitution of formulae with ${!}$, but, on the other hand, only a limited
version of the cut rule.

\section{The 3rd Approach: $\ELm$}\label{S:ELm}

In order to restore type substitution as much as possible we
consider the third approach to imposing Lambek's restriction on $\ELS$.
The trade-off here is that the cut rule is going to be admissible only
in a limited form.

We present such a system in the form of {\em marked sequent calculus.}
A marked sequent is an expression of the form $\Pi \to A$, where
$A$ is a type and $\Pi$ is a sequence of pairs of the form
$\langle B, m\rangle$, written as $\mk{B}{m}$, where $B$ is a type and
$m \in \{ 0, 1 \}$ is the
{\em marking bit}. A pair
$\mk{B}{0}$ is called an {\em unmarked type}, % and denoted just by $B$,
and $\mk{B}{1}$ is called a {\em marked type}. The marking bits are utilized inside the derivation,
and in the end they are forgotten, yielding a sequent in the original sense.
If $\Gamma = \mk{B_1}{m_1}, \dots, \mk{B_k}{m_k}$, then by
${!}\Gamma$ we denote the sequence
$\mk{({!}B_1)}{m_1}, \dots, \mk{({!}B_k)}{m_k}$.

Lambek's restriction is now formulated as follows:
{\em every sequent should contain an unmarked type in the antecedent.}

The calculus $\ELm$ is defined in the following way:

$$
\mk{p}{0} \to p
$$

$$
\infer[(\to\SL)\mbox{, where $\Pi$ contains an 
unmarked type}]{\Pi \to B \SL A}{\Pi, \mk{A}{m} \to B}
$$

$$
\infer[(\to\BS)\mbox{, where $\Pi$ contains an
unmarked type}]{\Pi \to A \BS B}{\mk{A}{m}, \Pi \to B}
$$

$$
\infer[(\SL\to)]{\Delta_1, \mk{(B \SL A)}{m}, \Pi, \Delta_2 \to C}
{\Pi \to A & \Delta_1, \mk{B}{m}, \Delta_2 \to C}
\qquad
\infer[(\BS\to)]{\Delta_1, \Pi, \mk{(A \BS B)}{m}, \Delta_2 \to C}
{\Pi \to A & \Delta_1, \mk{B}{m}, \Delta_2 \to C}
$$

$$
\infer[({!}\to)\mbox{, where $\Delta_1, \Delta_2$ contains
an unmarked type}]{\Delta_1, \mk{(!A)}{1}, \Delta_2 \to B}{\Delta_1, \mk{A}{m}, 
\Delta_2 \to B}
$$

$$
\infer[(\to{!})]{{!}\Gamma, {!}\Delta \to {!}A}{{!}\Gamma, \Delta \to A}
\qquad
\infer[(\WEAK)]{\Delta_1, \mk{({!}A)}{1}, \Delta_2 \to A}
{\Delta_1, \Delta_2 \to A}
$$

$$
\infer[(\CONTR)]{\mk{({!}A)}{\min\{ m_1, m_2 \}}, \Delta \to B}
{\mk{({!}A)}{m_1}, \mk{({!}A)}{m_2}, \Delta \to B}
$$

$$
\infer[(\EXC_1)]{\Delta_1, \mk{({!}A)}{m_1}, \mk{B}{m_2}, \Delta_2 \to C}
{\Delta_1, \mk{B}{m_2}, \mk{({!}A)}{m_1}, \Delta_2 \to C}
\qquad
\infer[(\EXC_2)]{\Delta_1, \mk{A}{m_1}, \mk{({!}B)}{m_2}, \Delta_2 \to C}
{\Delta_1, \mk{({!}B)}{m_2}, \mk{A}{m_1}, \Delta_2 \to C}
$$

Recall that all proofs are cut-free.
Also note that in $\EL$ we use a stronger form of the $(\to{!})$ rule.
In $\ELS$ this new rule could be simulated by applying the $({!}\to)$ rule 
for all formulae in $\Delta$ and then using the original $(\to{!})$ rule,
but here the $({!}\to)$ rule will fail to satisfy the restriction.

The substitution property is now formulated as follows:

\begin{theorem}
Let $A[q := Q]$ (resp., $\Pi[q := Q]$)
be the result of substituting $Q$ for $q$ in type $A$ (resp., marked sequence $\Pi$). 
Then  $\ELm \vdash \Pi \to A$ implies
$\ELm \vdash \Pi[q := Q] \to A[q := Q]$.
\end{theorem}

\begin{proof}
By structural induction on $Q$ we prove that
$\ELm \vdash \mk{Q}{0} \to Q$ for every type $Q$. Then
we just replace $q$ with $Q$ everywhere in the proof. \qed
\end{proof}

The cut rule in $\ELm$ is generally not admissible:
the sequents $\mk{({!}q)}{0} \to (p \SL {!}q) \BS p$ and
$\mk{((p \SL {!}q) \BS p)}{0} \to p \BS p$ are derivable in
$\ELm$, but $\mk{({!}q)}{0} \to p \BS p$ is not. This counterexample is
actually taken from the proof of Theorem~\ref{Th:ELd}.

The cut rule is admissible  only in the following limited version:

\begin{theorem}
If $\ELm \vdash \Pi \to A$, $\ELm \vdash \Delta_1, \mk{A}{0}, \Delta_2 \to C$, and
$A$ does not contain ${!}$,
then $\ELm \vdash \Delta_1, \Pi, \Delta_2 \to C$.
\end{theorem}

This theorem is proved using the standard argument, just as for $\LL$.

Compare $\ELm$ with $\ELs$. These two systems are not connected
with any strong form of conservativity or equivalence:
on one hand, the sequent ${!}r, r \BS {!}p, {!}(p \BS q) \to q$ is derivable
in $\ELs$, but not in $\ELm$; on the other hand,
 for ${!}p, {!}({!}p \BS q) \to q$ the situation is opposite.
 Fortunately,
the following holds:

\begin{lemma}\label{Lm:conserv}
If $\Gamma$, $\Pi$, and $A$ do not contain ${!}$, then
$$\ELm \vdash {!}\Gamma, \Pi \to A \iff \ELs \vdash {!}\Gamma, \Pi \to A.$$
\end{lemma}

\begin{proof}
Since for a sequent of the form
${!}\Gamma, \Pi \to A$ the rule $(\to{!})$ can never appear in the proof,
marked types in the antecedent are exactly the types starting with $!$, and
the two versions of Lambek's restriction coincide. \qed
\end{proof}

\section{Conservativity over $\LL$}

The three calculi defined above are conservative over $\LL$:

\begin{proposition}
If $\Pi$ and $A$ do not contain $!$, then
$$
\LL \vdash \Pi \to A
\iff
\ELw \vdash \Pi \to A
\iff
\ELs \vdash \Pi \to A
\iff
\ELm \vdash \Pi \to A
$$
(for $\ELm$, all types in $\Pi$ get the 0 marking bit).
\end{proposition}
Note that $\Pi$ is necessarily non-empty.

Therefore, we guarantee that in all approaches
the innovation affects only the new exponential connective,
and keeps the original Lambek system intact.
For $\ELs$ and $\ELm$ adding fresh exponentials to
the antecedent also does not affect Lambek's restriction:

\begin{proposition}
If $\Pi$ and $A$ do not contain ${!}$, and $p$ is a variable not
occurring in $\Pi$ and $A$, then $$\ELs \vdash {!}p, \Pi \to A
 \iff \ELm \vdash {!}p, \Pi \to A \iff \LL \vdash \Pi \to A$$ (for the
 $\ELm$ case, ${!}p$ gets marking bit 1 and types from $\Pi$ get 0).
 \end{proposition}

For $\ELw$, due to Lemma~\ref{Lm:vanish}, the situation is different:
if $\Pi$ and $A$ do not contain ${!}$, and $p$ is a fresh variable, then
$$\ELw \vdash {!}p, \Pi \to A \iff  \ELS \vdash \Pi \to A \iff \LLS \vdash \Pi \to A.$$ Recall that, for example, $(q \BS q) \BS p \to p$ is derivable in $\LLS$, but not
in $\LL$.

\section{Generative Grammars and the Lambek Calculus with Non-Logical Axioms}\label{S:Buszko}

In this subsection we introduce {\em axiomatic extensions}
of the Lambek calculus $\LL$, following~\cite{Buszko82}.
These extensions are going to be useful for proving
undecidability results \`{a} la Theorem~\ref{Th:undecELS}.

Let
$\Ac$ be a set of sequents. Then by $\LL + \Ac$ we denote
$\LL$ augmented with sequents from $\Ac$ as new axioms and
also the cut rule (which is no longer eliminable). Elements
of $\Ac$ are called {\em non-logical axioms.} 

Further we consider non-logical axioms of a special form:
either $p, q \to r$, or $p \SL q \to r$, where $p,q,r$ are variables.
Buszkowski calls them {\em special} non-logical axioms.
In this case, $\LL + \Ac$ can be formulated in a cut-free 
way~\cite{Buszko82}: instead of non-logical axioms of the
form $p, q \to r$ or $p \SL q \to r$ we use rules
$$\infer[(\RED_1)]{\Pi_1, \Pi_2 \to r}{\Pi_1 \to p & \Pi_2 \to q}
\mbox{\qquad and\qquad}
\infer[(\RED_2)\mbox{, where $\Pi \ne \Lambda$}]{\Pi \to r}{\Pi, q \to p}$$ respectively. This calculus admits the cut rule~\cite{Buszko82}. Further we'll mean it when
talking about $\LL + \Ac$. We'll use the term {\em Buszkowski's rules}
for $(\RED_i)$.

Now we define two notions of {\em formal grammar}. The first one
is the widely known formalism of {\em generative grammars} introduced by Chomsky.
If $\Sigma$ is an alphabet (i.e.\ a finite non-empty set), then by $\Sigma^*$ we
denote the set of all words over $\Sigma$ (including the empty word).
A generative grammar is a quadruple $G = \langle N, \Sigma, s, P \rangle$, where
$N$ and $\Sigma$ are two disjoint alphabets, $s \in N$, and $P$ is a set or {\em rules.}
Here we consider only rules of two forms: $x \to y_1 y_2$ or $x_1 x_2 \to y$, where
$x, y, x_i, y_i \in N \cup \Sigma$. If $v = u_1 \alpha u_2$, $w = u_1 \beta u_2$, and
$(\alpha \to \beta) \in P$, then this rule can be {\em applied} to $v$ yielding $w$:
$v \Rightarrow w$. By $\Rightarrow^*$ we denote the reflexive and transitive
closure of $\Rightarrow$. Finally, the {\em language generated by $G$} is the set of
all words $w \in \Sigma^*$, such that $s \Rightarrow^* w$. Note that
the empty word cannot be produced by a generative grammar as \mbox{defined above.}

It is well known that the class of languages generated by generative grammars
coincides with the class of all recursively enumerable (r.\,e.) languages
without the empty word.

The second family of formal grammar we are going to consider is the class
of {\em Lambek categorial grammars with non-logical axioms.} A Lambek grammar
is a tuple $\mathcal{G} = \langle \Sigma, \Ac, H, \rhd \rangle$, where
$\Sigma$ is an alphabet, $\Ac$ is a set of non-logical axioms, $H$ is a type,
and $\rhd \subseteq \Tp \times \Sigma$ is a finite binary correspondence between
types and letter, called {\em type assignment.}
A word $w = a_1 \dots a_n$ belongs to the language generated by
$\mathcal{G}$ if{f} there exist such types $A_1, \dots, A_n$ that
$A_i \rhd a_i$ ($i = 1, \dots, n$) and $\LL + \Ac \vdash A_1, \dots, A_n \to H$.

If we use $\LL\!_{{}\SL{}}$ instead of $\LL$, we get the notion of
{\em $\LL\!_{{}\SL{}}$-grammar} with non-logical axioms.
It's easy to see that all languages generated by Lambek grammars are r.\,e., therefore,
they can be generated by generative grammars. Buszkowski~\cite{Buszko82} proves the
converse:

\begin{theorem}\label{Th:Buszko}
Every language generated by a generative grammar can be generated by
an $\LL\!_{{}\SL{}}$-grammar with special non-logical axioms.
\end{theorem}

In comparison, for $\Ac = \varnothing$ Pentus' theorem~\cite{Pentus93} states
that all languages generated are context-free.
 Thus, even simple (special) non-logical axioms dramatically increase
the power (and complexity) of Lambek grammars.

Since there exist undecidable r.\,e.\ languages, Buszkowski obtains the 
following~\cite{Buszko82}:
\begin{theorem}\label{Th:LAundec}
There exists such $\Ac$ that the derivability problem for $\LL\!_{{}\SL{}} + \Ac$
is undecidable.
\end{theorem}

\section{Undecidability of $\ELs$ and $\ELm$}\label{S:undec}

Recall that $\ELs$, defined in Section~\ref{S:ELs}, involves
two division operations, and no product. The calculus $\ELsu$
is the fragment of $\ELs$, where we confine ourselves only
to the right division.

\begin{theorem}\label{Th:undecELsu}
The derivability problem for $\ELs$ and even for $\ELsu$ is undecidable.
\end{theorem}

We take a set $\Ac$ of non-logical axioms of non-logical axioms
of the forms $p, q \to r$ or $p \SL q \to r$
and encode them
in $\ELs$ using the exponential. 
Let $\Gc_{\Ac} = \{ (r \SL q) \SL p \mid
(p,q \to r) \in \Ac \} \cup \{ r \SL (p \SL q) \mid
(p \SL q \to r) \in \Ac \}$ and let $\Gamma\!_\Ac$ be a
sequence of all types from $\Gc_\Ac$ in any order. Then the following holds:

\begin{lemma}\label{Lm:AxToExp}
$\LL_{{}\SL{}} + \Ac \vdash \Pi \to A \iff
\ELsu \vdash {!} \GA, \Pi \to A$.
\end{lemma}

\begin{proof}

\fbox{$\Rightarrow$} 
Proceed by induction on the derivation of $\Pi \to A$ in $\LLu + \Ac$.
If $\Pi \to A$ is an axiom of the form $A \to A$, then we get 
$\ELsu \vdash {!}\GA, A \to A$ by application of the  $(\WEAK)$ rule.

If $A = B \SL C$, and $\Pi \to A$ is obtained using the $(\to\SL)$ rule,
then \mbox{$!\GA, \Pi \to A$} is derived using the same rule:
$$
\infer{!\GA, \Pi \to B \SL C}{!\GA, \Pi, C \to B}
$$

Here $\Pi$ is not empty, and consists of non-bang-formulae,
therefore the application of this rule is eligible in $\ELsu$;
$\ELsu \vdash {!}\GA, \Pi, C \to B$ by induction hypothesis.

If $\Pi = \Phi_1, B \SL C, \Psi, \Phi_2$, and $\Pi \to A$ is
obtained by $(\SL\to)$ from $\Psi \to C$ and $\Phi_1, B, \Phi_2 \to A$, then
for $!\GA, \Pi \to A$ we have the following derivation in $\ELsu$, where
${}^*$ means several applications of the rules in any order.

$$
\infer[(\CONTR, \EXC_1)^*]{!\GA, \Phi_1, B \SL C, \Psi, \Phi_2 \to A}
{\infer[(\EXC_1)^*]
{!\GA, !\GA, \Phi_1, B \SL C, \Psi, \Phi_2 \to A}
{
\infer[(\SL\to)]
{!\GA, \Phi_1, B \SL C, !\GA, \Psi, \Phi_2 \to A}
{!\GA, \Psi \to C &
!\GA, \Phi_1, B, \Phi_2 \to A}
}}
$$

Finally, $\Pi \to A$ can be obtained by application of
Buszkowski's rules $(\RED_1)$ or $(\RED_2)$. In the first case,
$A = r$, $\Pi = \Pi_1, \Pi_2$;  $\LL + \Ac \vdash \Pi_1 \to p$, and
$\LL + \Ac \vdash \Pi_2 \to q$. Furthermore, $\Gc_{\Ac} \ni
(r \SL q) \SL p$, thus we get the following derivation in $\ELsu$:

$$
\infer[(\CONTR, \EXC_1)^*]
{!\GA, \Pi_1, \Pi_2 \to r}
{\infer[(\CONTR, \EXC_1)^*]
{!\GA, {!}((r \SL q) \SL p), \Pi_1, \Pi_2 \to r}
{\infer[({!}\to)]
{ {!}((r \SL q) \SL p), !\GA, \Pi_1, !\GA, \Pi_2  \to r}
{\infer[(\SL\to)]
{(r \SL q) \SL p, !\GA, \Pi_1, !\GA, \Pi_2  \to r}
{!\GA, \Pi_1 \to p & 
\infer[(\SL\to)]{r \SL q, !\GA, \Pi_2 \to r}{!\GA, \Pi_2 \to q & r \to r}
}}}}
$$

The application of $({!}\to)$ here is legal, since $\Pi_1$ and $\Pi_2$ are
non-empty and consist of non-bang-formulae.

In the $(\RED_2)$ case, $A = r$, and we have ${!}\GA, \Pi, q \to p$ in
the induction hypothesis. Again, $\Gc_{\Ac} \ni r \SL (p \SL q)$, and we proceed
like this:

$$
\infer[(\CONTR, \EXC_1)^*]{{!}\GA, \Pi \to r}
{\infer[({!}\to)]
{{!}(r \SL (p \SL q)), {!}\GA, \Pi \to r}
{\infer[(\SL\to)]
{r \SL (p \SL q), {!}\GA, \Pi \to r}
{\infer[(\to\SL)]
{{!}\GA, \Pi \to p \SL q}
{{!}\GA, \Pi, q \to p}
 & r \to r}
}}
$$

Here, again, $\Pi$ is not empty and consists of non-bang-formulae, therefore
we can legally apply $({!}\to)$ and $(\to\SL)$.

\fbox{$\Leftarrow$}  
For deriving sequents of the form ${!}\Gamma, \Pi \to A$, where $\Gamma$,
$\Pi$, and $A$ do not contain the exponential, one can use a simpler calculus
than $\ELsu$:
$$
{!}\Gamma, p \to p
$$
$$
\infer[(\to\SL)\mbox{, where $\Pi \ne\Lambda$}]{{!}\Gamma, \Pi \to A \SL B}{{!}\Gamma, \Pi, B \to A}
\qquad
\infer[(\SL\to)]{{!}\Gamma, \Delta_1, A \SL B, \Pi, \Delta_2 \to C}
{{!}\Gamma, \Pi \to B & {!}\Gamma, \Delta_1, A, \Delta_2 \to C}
$$
$$
\infer[({!}\to)\mbox{, where $B$ is a type from $\Gamma$ and $\Delta_1, \Delta_2 \ne \Lambda$}]
{{!}\Gamma, \Delta_1, \Delta_2 \to A}{{!}\Gamma, \Delta_1, B,
\Delta_2 \to A}
$$

Here $(\WEAK)$ is hidden into the axiom, $(\CONTR)$ comes within
$(\to\SL)$, and $({!}\to)$ includes both $(\EXC_i)$ and $(\CONTR)$ in the
needed form. One can easily see that if $\ELsu \vdash {!}\Gamma, \Pi \to A$,
where $\Gamma$, $\Pi$, and $A$ do not contain ${!}$, then this sequent is
derivable in the simplified calculus. Moreover, the $({!}\to)$ rule is
interchangeable with the others in the following ways:

$$
\infer[({!}\to)]{{!}\Gamma, \Delta_1, \Delta_2 \to A \SL B}
{\infer[(\to\SL)]
{{!}\Gamma, \Delta_1, C, \Delta_2 \to A \SL B}
{{!}\Gamma, \Delta_1, C, \Delta_2, B \to A}
}
\qquad
\mbox{\raisebox{1.2em}{$\leadsto$}}
\qquad
\infer[(\to\SL)]{{!}\Gamma, \Delta_1, \Delta_2 \to A \SL B}
{\infer[({!}\to)]
{{!}\Gamma, \Delta_1, \Delta_2, B \to A}
{{!}\Gamma, \Delta_1, C, \Delta_2, B \to A}
}
$$

\vskip 10pt

$$
\infer[({!}\to)]{{!}\Gamma, \Delta_1, A \SL B, \Pi, \Delta'_2, \Delta''_2 \to C}
{\infer[(\SL\to)]
{{!}\Gamma, \Delta_1, A \SL B, \Pi, \Delta'_2, D, \Delta''_2 \to C}
{{!}\Gamma, \Pi \to B & {!}\Gamma, \Delta_1, A, \Delta'_2, D, \Delta''_2 \to C}
}
$$
\centerline{\turnbox{270}{$\leadsto$}}
\vskip 10pt
$$
\infer[(\SL\to)]{{!}\Gamma, \Delta_1, A \SL B, \Pi, \Delta'_2, \Delta''_2 \to C}
{{!}\Gamma, \Pi \to B & \infer[({!}\to)]
{{!}\Gamma, \Delta_1, A, \Delta'_2, \Delta''_2 \to C}
{{!}\Gamma, \Delta_1, A, \Delta'_2, D, \Delta''_2 \to C}
}
$$
And the same, if $D$ appears inside $\Delta_1$ or $\Pi$. Finally, consecutive
applications of $({!}\to)$ are always interchangeable.

After applying these transformations, we achieve a derivation where
$({!}\to)$ is applied immediately after applying $(\SL\to)$ with the same
active type (the other case, when it is applied after the axiom to $p$, is
impossible, since then it violates the non-emptiness condition).
In other words, applications of $({!}\to)$ appear only in the following
two situations:
$$
\infer[({!}\to)]{{!}\Gamma, \Delta_1, \Pi, \Delta_2 \to A}
{
\infer[(\SL\to)]{{!}\Gamma, \Delta_1, (r \SL q) \SL p, \Pi, \Delta_2 \to A}
{{!}\Gamma, \Pi \to p & {!}\Gamma, \Delta_1, r \SL q, \Delta_2 \to A}
}
$$
and
$$
\infer[({!}\to)]{{!}\Gamma, \Delta_1, \Pi, \Delta_2 \to A}
{\infer[(\SL\to)]{{!}\Gamma, \Delta_1, r \SL (p \SL q), \Pi, \Delta_2 \to A}
{{!}\Gamma, \Pi \to p \SL q & {!}\Gamma, \Delta_1, r, \Delta_2 \to A}
}
$$

Now we prove the statement $\ELsu \vdash {!}\GA, \Pi \to A \Rightarrow
\LL + \Ac \vdash \Pi \to A$ by induction on the above canonical derivation.
 If ${!}\GA, \Pi \to A$ is an axiom or is obtained by an application of $(\SL\to)$ or
$(\to\SL)$, we apply the corresponding rules in $\LL + \Ac$, so the only interesting case
is $({!}\to)$. Consider the two possible situations.

In the $(r \SL q) \SL p$ case, by induction hypothesis we get 
$\LL + \Ac \vdash \Pi \to p$ and $\LL + \Ac \vdash \Delta_1, r \SL q, \Delta_2 \to A$,
and then we develop the following derivation
in $\LL+\Ac$ (recall that $(\CUT)$ is admissible there):
$$
\infer[(\CUT)]{\strut\Delta_1, \Pi, \Delta_2 \to A}
{\Pi \to p & 
\infer[(\CUT)]{\strut \Delta_1, p, \Delta_2 \to A}
%\infer[(\to\SL)]{
{\strut \infer[(\to\SL)]{p \to r \SL q}{p,q\to r} &
%}{\infer[(\RED_1)]{\strut p,q \to r}{p \to p & q \to q}} & 
\Delta_1, r \SL q, \Delta_2 \to A}
}
$$

In the case of $r \SL (p \SL q)$, the derivation looks like this:
$$
\infer[(\CUT)]{\strut \Delta_1, \Pi, \Delta_2 \to A}
{\Pi \to p \SL q & 
\infer[(\CUT)]{\strut \Delta_1, p \SL q, \Delta_2 \to A}
{
%\infer[(\RED_2)]{
p \SL q \to r
%}{\infer[(\SL\to)]{p \SL q, q \to p}{q \to q & p\to p}} 
& \Delta_1, r, \Delta_2 \to A}}
$$
\qed
\end{proof}

Note that in this proof we do not need any form
of the cut rule for $\ELs$.

\begin{theorem}
The derivability problem for $\ELm_{\SL}$ (and, thus, for $\ELm$)
is undecidable.
\end{theorem}

\begin{proof}
By Lemma~\ref{Lm:conserv} we have that if $\Gamma$, $\Pi$, and $A$ do not contain ${!}$, then
$\ELm_{\SL} \vdash {!}\Gamma, \Pi \to A \iff \ELsu \vdash
{!}\Gamma, \Pi \to A$. 
\qed
\end{proof}

Of course, everything discussed above can be dually performed for
$\BS$ instead of $\SL$, yielding undecidability for $\ELm_{\BS}$
and $\ELs\!_{\BS}$.

\section{Conclusion}\label{S:conclusion}

The derivability problem for the original Lambek calculus,
without exponential modalities, is decidable and belongs
to the NP class. This happens because the cut-free proof of
a sequent has linear size with respect to the sequent's length.
For the full Lambek calculus~\cite{Pentus06} and for its fragments
with any two of three connectives (two divisions~\cite{Savateev09LFCS}
 or one division and
the product~\cite{Savateev09PhD}) the derivability problem
is NP-complete. 

On the other hand, for derivability problem in $\LL_{\SL}$
there exists a polynomial time algorithm~\cite{Savateev08CSR}.
Thus the one-division fragment of the Lambek calculus appears
to be significantly simpler.
Despite this, in our undecidability results for $\ELs$ and
$\ELm$ we use only one
of the two divisions.

\subsection*{Related Work}
%FIXME We also plan to investigate calculi with modalities where not all of the structural 
%rules ($(\WEAK)$, $(\EXC_i)$, and $(\CONTR)$) are kept. 
%Once we remove $(\CONTR)$, the derivability problem becomes decidable and
%falls into the NP class. The interesting question is to determine
%precise complexity bounds (P or NP-hard) for the fragments with only one
%division and bang (if we have at least two of the three Lambek connectives,
%even the calculus without bang is NP-complete).
%The variants where we have only permutational rules (both $(\EXC_1)$ and $(\EXC_2)$, 
%or only one of them) are particularly interesting for linguistic applications
%(see, for example,~\cite{MootRetore12,Morrill90}).
%The other question here is (un)decidability of the variant
%without $(\EXC_i)$ (only $(\WEAK)$ and $(\CONTR)$).

In the case of commutative linear logic  Nigam and Miller~\cite{Nigam09} consider calculi
that have several modalities interacting with each other, and
different modalities are controlled by different sets of
structural rules. These modalities are called 
{\em subexponentials}.
A systematic study of subexponentials in the non-commutative case, under the umbrella of the Lambek
calculus and cyclic linear logic, is performed in~\cite{KanKuzNigSce18MSCS}.
We also plan to study systems with (sub)exponentials, built on top of light~\cite{Girard98} and
soft~\cite{Lafont04} linear logic. Some preliminary results in this direction were presented at
the 2018 Mal'tsev Meeting~\cite{KanKuzSce18Maltsev}.

It appears that the technique used in the \fbox{$\Leftarrow$} part
of the proof of Lemma~\ref{Lm:AxToExp} is an non-commutative instance of
{\em focusing}~\cite{Andreoli92,Nigam08}.
Focusing techniques in the non-commutative case, for the Lambek calculus extended
with subexponentials, are developed in~\cite{KanKuzNigSce18IJCAR}.

\subsection*{Acknowledgments}
%Stepan Kuznetsov's research was supported by the Russian Foundation
%for Basic Research 
%(grants 15-01-09218-a and 14-01-00127-a)
%and by the Presidential Council for Support of Leading
%Scientific Schools
%(grant N\v{S}-9091.2016.1). Max Kanovich's research was partially supported
%by EPSRC. Andre Scedrov's research was partially supported by ONR.
This research was performed in part during visits of Stepan Kuznetsov and Max Kanovich
to the University of Pennsylvania. We greatly appreciate support of the
Mathematics Department of the University.
A part of the work was also done during the stay of Andre Scedrov at
the National Research University Higher School of Economics.
 We would like to thank 
S.~O.~Kuznetsov and I.~A.~Makarov for hosting us there.

%The paper was prepared in part within the framework of the Basic
%Research Program at the National Research University Higher
%School of Economics (HSE) and was partially supported within the
%framework of a subsidy by the Russian Academic Excellence Project
%`5--100'. 

This article was prepared within the framework of the HSE University Basic Research Program and funded by the Russian Academic Excellence Project `5-100.'

%A preliminary version of this paper was presented at the
%Logical Foundations of Computer Science 2016 symposium.
%We are grateful to the participants of LFCS for fruitful
%discussions.

%\label{}





\section*{References}

\bibliography{EL_APAL}

\iffalse

\fi
\end{document}